\documentclass[a4paper,12pt]{amsart}
\usepackage{amsmath,amssymb,amsthm,amsxtra,mathrsfs,verbatim}
\usepackage{mathtools}
\usepackage{a4wide}

\newcommand{\C}{{\mathbb C}}
\newcommand{\N}{{\mathbb N}}
\newcommand{\R}{{\mathbb R}}
\newcommand{\Z}{{\mathbb Z}}

\newcommand{\abs}[2][\empty]{\ifx#1\empty\left|#2\right|%
\else#1\vert #2 #1\vert\fi}
\newcommand{\Cnt}[1][]{{\mathcal C}^{#1}}
\newcommand{\co}{\mathop{\mathrm{co}}}
\newcommand{\csub}{\subset\subset}
\newcommand{\defstyle}[1]{\emph{#1}}
\newcommand{\eps}{\varepsilon}
\newcommand{\ext}{\mathop{\mathrm{ext}}\nolimits}
\newcommand{\Fourier}{\mathcal F}
\newcommand{\fourier}{\widehat}
\newcommand{\Idem}{\mathrm{Id}}
\renewcommand{\implies}{\Rightarrow}
\newcommand{\interl}{\mathop{\mathrm{interl}}}
\renewcommand{\land}{\mathrel{\&}}
\newcommand{\norm}[2][\empty]{\ifx#1\empty\left\Vert#2\right\Vert%
\else#1\Vert #2 #1\Vert\fi}
\newcommand{\Powerset}{{\mathcal P}_\emptyset}
\newcommand{\supp}{\mathop{\mathrm{supp}}}
\newcommand{\Sphere}{\mathbf S}
\newcommand{\Schwartz}{\mathscr{S}}
\newcommand{\test}{\mathcal D}

\newcommand{\Gen}{{\mathcal G}}
\newcommand{\GenC}{\widetilde\C}
\newcommand{\GenR}{\widetilde\R}
\newcommand{\Mod}{{\mathcal M}}
\newcommand{\Null}{{\mathcal N}}
\newcommand{\caninf}{\rho}
\newcommand{\slow}{\mathit{slow}}
\newcommand{\fast}{\mathit{fast}}

\newcommand{\ster}[1]{\prescript{*}{}{#1}}
\newcommand{\sterco}{\mathop{\ster\co}}
\newcommand{\stercup}{\mathrel{\ster\cup}}
\newcommand{\sternotin}{\mathrel{\ster\notin}}
\newcommand{\sterne}{\mathrel{\ster\ne}}

\newtheorem{theorem}{Theorem}[section]
\newtheorem{lemma}[theorem]{Lemma}
\newtheorem{proposition}[theorem]{Proposition}
\newtheorem{corollary}[theorem]{Corollary}

\theoremstyle{definition}
\newtheorem{definition}[theorem]{Definition}
\newtheorem{remark}[theorem]{Remark}
\newtheorem{notation}[theorem]{Notation}

\numberwithin{equation}{section}

\begin{document}
\title[Internal objects and Microlocal analysis]
{An application of internal objects to microlocal analysis in generalized function algebras}
\author{H.~Vernaeve}
\thanks{Supported by grant 1.5.138.13N of the Research Foundation Flanders FWO}

\begin{abstract}
We illustrate the use of internal objects in the nonlinear theory of generalized functions by means of an application to microlocal analysis in Colombeau algebras.
\end{abstract}
\maketitle

\section{Introduction}
As is well known, the Colombeau algebras $\Gen(\Omega)$ \cite{Colombeau84,Colombeau85,GKOS} are differential algebras of generalized functions containing the space of Schwartz distributions. They have found diverse applications in the study of partial differential equations \cite{Hormann-DeHoop,Hormann-Ober-Pilipovic,Ob92}, providing a framework in which nonlinear equations and equations with strongly singular data or coefficients can be solved and in which their regularity can be analyzed.

The natural extension of microlocal analysis of Schwartz distributions to the Colombeau generalized function algebras is the so-called $\Gen^\infty$-microlocal analysis, which has been developed using the concept of $\Gen^\infty$-regularity \cite{Dapic98,GGO05,Garetto-Hormann05,Garetto-Ober14,Hormann-Ober-Pilipovic}. Recently, we introduced a refinement called $\widetilde\Gen^\infty$-microlocal analysis in which one can consider microlocal regularity at generalized points $(x_0,\xi_0)$ in the cotangent bundle of the domain \cite{HVMicrolocal}. Also the neighbourhoods which can be considered have generalized (infinitesimal) radii. The reason to introduce this refinement is the following: since equations with strongly singular data or coefficients in Colombeau algebras are modeled by regularization, the corresponding differential operators themselves become generalized operators. Hence it is to be expected that the most suitable setting to study the propagation of singularities under such operators is by means of generalized objects (generalized characteristic varieties, etc.).

The development of $\widetilde\Gen^\infty$-microlocal analysis has been obtained using the principles introduced in \cite{HVNSPrinciples,HV-NSPrinciples-tutorial} which originate from nonstandard analysis, although \cite{HVMicrolocal} does not require the knowledge of these principles from the reader. That writing style comes with two disadvantages: a loss of intuition, and rather technical proofs. In this paper, we show that these principles are developed well enough to write the complete proofs of these recent research results, thus revealing the underlying intuition more clearly. At several places (especially in Theorem \ref{thm-projection-of-wave-front-2}), we were able to significantly simplify the proofs in \cite{HVMicrolocal}. We hope that this may serve as an example for researchers in the field, helping them to use the same techniques also in their own work.

To keep the paper self-contained, we recall the definitions of the internal objects and the principles that are used.

\section{Internal objects}
The internal objects that we will consider are closely related to the approach in \cite{Schmieden-Laugwitz} (see also \cite{Tao}).

Let $\emptyset\ne A\subseteq \R^d$ and $I:= (0,1]\subseteq\R$. We denote
\[\ster A:= {A^ I}_{/\sim}\]
where $\sim$ is the equivalence relation defined by
\[(x_\eps)_{\eps\in I}\sim (y_\eps)_{\eps\in I} \iff (\exists \eps_0) (\forall \eps\in (0,\eps_0]) (x_\eps = y_\eps)\] which we read as: $x_\eps=y_\eps$ for (sufficiently) small $\eps$. We denote by $[x_\eps]\in \ster A$ the equivalence class of $(x_\eps)_\eps\in A^I$. By definition, elements of $\ster A$ are called internal. We denote $\ster a:= [a]$ (the equivalence class of the constant family $(a)_{\eps\in I}$). Since this defines an injection $\ster{}$: $A\to\ster A$, we will identify $a\in A$ with $\ster a\in\ster A$. It is clear that $\ster (\R^d)$ and $(\ster \R)^d$ are isomorphic, and we will identify both.

For any map $f$: $A\to\R^{d'}$, there is a canonical extension
\begin{equation}\label{eq-generalized-function}
\ster f: \ster A\to\ster\R^{d'}: f([x_\eps]):= [f(x_\eps)].
\end{equation}
Since it is a canonical extension, it is customary to write $f:= \ster f$.

If $\emptyset\ne A_\eps\subseteq\R^d$ (for each $\eps\in I$), we define a set
\[[A_\eps]:= \{[x_\eps]: x_\eps\in A_\eps, \text{ for small }\eps\}\subseteq\ster\R^d.\]
By definition, such subsets are called internal. In particular, $\ster A= [A]$ is internal.

More generally than in equation \eqref{eq-generalized-function}, if $\emptyset\ne A_\eps\subseteq\R^d$ and $f_\eps$ are maps $A_\eps\to \R^{d'}$ (for each $\eps\in I$), we define a map
\[[f_\eps]: [A_\eps]\to\ster\R^{d'}:  [f_\eps]([x_\eps]):= [f_\eps(x_\eps)].\]
By definition, such maps are called internal.

Any binary relation $R$ on $A$ has an extension on $\ster A$ (which is also called internal):
\[[x_\eps]\, \ster R\, [y_\eps] \iff x_\eps \, R\, y_\eps, \text{ for small }\eps.\]
Since it is a canonical extension, it is customary to write $R:= \ster R$. Caution: $ x\ster (\lnot R) y$ is not equivalent with $\lnot( x  \mathop{\ster R} y)$, hence we will not drop stars in $\sterne$, $\ster{\not\le}$, \dots to avoid ambiguities (in contrast, we use $x\ne y$ for $\lnot(x=y)$).

If $X_\eps$ are nonempty sets of maps $A\to\R^{d'}$ (for each $\eps\in I$), we define
\[[X_\eps] := \{[f_\eps]: f_\eps\in X_\eps, \text{ for small }\eps\}.\]
By definition, $[X_\eps]$ consists of internal maps only. We define again $\ster X:= [X]$. Then for any $f\in X$, its canonical extension $\ster f$ belongs to $\ster X$. In this paper, we will mainly consider $\ster X$ for a function space $X\subseteq\Cnt[\infty](\R^d)$.

We denote by $\Powerset(A)$ the set of all \emph{nonempty} subsets of $A$. If $\emptyset\ne B\subseteq \Powerset(\R^d)$, we define
\[\ster B:= \{[A_\eps]: A_\eps\in B,\text{ for small }\eps\}.\]
Thus $\ster(\Powerset(\R^d))$ is the set of all internal subsets of $\ster \R^d$ as defined above (notice that $\emptyset$ is not internal by definition).

More generally, if (for each $\eps$) $f_\eps$ are maps $X_1\times \cdots\times X_m \to Y_1\times \cdots \times Y_n$, where each $X_i$ and $Y_j$ either is a nonempty subset of $\R^d$ or a nonempty set of maps $A\to \R^{d'}$ or a nonempty subset of $\Powerset(B)$ (where $d$, $d'$ and $A,B\subseteq\R^d$ may depend on $i$, $j$), we define a map
\[[f_\eps]: \ster X_1\times \cdots\times \ster X_m\to\ster Y_1\times \cdots\times \ster Y_n:  [f_\eps]([x_{1,\eps}], \dots, [x_{m,\eps}]):= [f_\eps(x_{1,\eps},\dots,x_{m,\eps})]\]
where we identify $\ster(Y_1\times \cdots\times Y_n)\cong \ster Y_1 \times \cdots\times \ster Y_n$, i.e.\ we identify
\[[f_\eps(x_\eps)] = \bigl( [f_{1,\eps}(x_\eps)], \dots, [f_{m,\eps}(x_\eps)]\bigr) .\]
By definition, also such maps are called internal. We define again $\ster f:= [f]$.

E.g., in this paper, we will use
\[A \stercup B,
\quad\mathop{\ster\sup} A,\quad \ster\partial^\alpha f,\quad \ster{\int_A} f \quad \text{and}\quad \ster\Fourier(f)\]
for internal $A$, $B$, $\alpha$, $f$, where we consider
\begin{align*}
\cup: &\Powerset(\R^d)\times\Powerset(\R^d)\to \Powerset(\R^d)&
\sup:& \{A\subseteq\R: A\ne\emptyset\text{ is bounded}\}\to\R\\
 \partial: &\N^d\times \Cnt[\infty](\R^d)\to \Cnt[\infty](\R^d)&
\int: &\{A\subseteq \R^d: A\ne\emptyset\text{ is measurable}\}\times L^1(\R^d)\to \C\\
\Fourier:&\Schwartz(\R^d)\to\Schwartz(\R^d).
\end{align*}
As differentiation, integration and Fourier transform are among the most basic operations in analysis, we will write $\partial:= \ster\partial$, $\int:= \ster\int$ and $\Fourier:= \ster\Fourier$ (this corresponds to the usual $\eps$-wise definitions in Colombeau theory). As $\ster\sup$ is the supremum for the partial order on $\ster\R$, we will write $\sup:= \ster\sup$. However, as $\ster\cup$ is not the set-theoretic union, we will not drop stars here.

\smallskip
The above construction can be inductively extended to larger classes of objects, all of which are called internal \cite{HVNSPrinciples}, but the above definitions suffice for this paper.

\section{Principles from nonstandard analysis}
Proving properties about internal objects is considerably simplified through the following principles \cite{HVNSPrinciples}.

\begin{definition}
An object $a$ is called transferrable if $\ster a$ is defined.

\emph{Transferrable formulas} are formal expressions containing symbols called \defstyle{variables}. Particular kinds of variables are \defstyle{relation variables} and \defstyle{function variables}.\\
Inductively, \defstyle{terms} are defined by the following rules:
\begin{enumerate}
\item A variable is a term.
\item If $t_1$, \dots, $t_m$ are terms ($m>1$), then also $(t_1, \dots, t_m)$ is a term.
\item If $t$ is a term and $f$ is a function variable, then also $f(t)$ is a term.
\end{enumerate}

Transferrable formulas are defined by the following rules:
\begin{enumerate}
\item[F1.] (atomic formulas) If $t_1$, $t_2$ are terms and $R$ is a relation variable, then $t_1=t_2$, $t_1\in t_2$ and $t_1 R\, t_2$ are formulas.
\item[F2.] If $P$, $Q$ are formulas, then $P\land Q$ is a formula.
\item[F3.] If $P$ is a formula, $x$ is a variable free in $P$ and $t$ is a term in which $x$ does not occur, then $(\exists x\in t) P$ is a formula.
\item[F4.] If $P$ is a formula, $x$ is a variable free in $P$ and $t$ is a term in which $x$ does not occur, then $(\forall x\in t) P$ is a formula.
\item[F5.]\label{rule_impl} If $P$, $Q$ are formulas, then $[(\exists x\in t)P]\,\&\, [(\forall x\in t) (P\implies Q)]$ is a formula.
\end{enumerate}
\end{definition}
In practice, we will use rule [F5] in formulas by simply writing $(\forall x\in t) (P\implies Q)$, silently checking that the side condition $(\exists x\in t) P$ is fulfilled.

\begin{notation}
We write $P(x_1,\dots, x_m)$ for a formula $P$ in which the only occurring free variables are $x_1$, \dots, $x_m$. We denote by $P(c_1,\dots, c_m)$ the formula $P$ in which the variable $x_j$ has been substituted by the object $c_j$ (for $j=1,\dots, m$). In this case, $c_j$ are called the constants occurring in the formula. Relation variables (resp.\ function variables) can only be substituted by (binary) relations (resp.\ functions).\\
The slight ambiguity that might result from these notations is similar to the notation for a function $f(x)$ and its value $f(c)$, and is clarified by the context.\\
An \emph{internal formula} is a transferrable formula in which all constants are internal.
\end{notation}

Notice that disjunction ($\lor$) and negation ($\lnot$) are not allowed in the formation rules. Also, we only allow bounded quantifiers (i.e., expressions `$\forall x$' and `$\exists x$' have to be followed by `$\in t$'). The reasons to consider this particular class of formulas are:

\begin{theorem}[Transfer Principle]\label{thm_up-down-transfer}\cite{HVNSPrinciples}
Let $P(a_1,\dots,a_m)$ be a transferrable formula, in which the constants $a_j$ are transferrable objects. Then
\[P(a_1, \dots, a_m)\text{ is true} \quad\text{iff}\quad P(\ster a_1, \dots, \ster a_m) \text{ is true}.\]
\end{theorem}
\begin{remark}
The transferrable formulas that we just defined are called h-formulas in \cite{Palyutin} and Palyutin formulas in \cite{Palmgren}. Already \cite{Palyutin} mentions the transfer principle for these formulas.
\end{remark}

\begin{theorem}[Internal Definition Principle]\label{thm_IDP}\cite{HVNSPrinciples}
Let $A$ be an internal set. Let $P(x)$ be an internal formula. If $\{x\in A: P(x)\}\ne\emptyset$, then $\{x\in A: P(x)\}$ is internal.
\end{theorem}

\begin{definition}
We call $x\in\ster\R$ infinitesimal if $\abs x\le 1/n$ for each $n\in\N$. Notation: $x\approx 0$.\\
We call $x\in\ster\R$ infinitely large if $\abs{x}\ge n$ for each $n\in\N$. Notation: $x\in\ster\R_\infty$. We also write $A_\infty:= A\cap \ster\R_\infty$ for $A\subseteq\ster\R$.
\end{definition}
Notice that $\N$, $\R$, $\ster\R_\infty$ and $\approx$ are external (i.e., not internal), and therefore not allowed in internal formulas.

\begin{theorem}[Spilling principles]\cite{HVNSPrinciples}
Let $A\subseteq\ster\N$ be internal.
\begin{enumerate}
\item (Overspill) If $\N\subseteq A$, then there exists $\omega\in\ster\N_\infty$ such that $\{n\in\ster\N: n\le \omega\}\subseteq A$.
\item (Underspill) If $\ster\N_\infty\subseteq A$, then $A\cap \N\ne\emptyset$.
\end{enumerate}
\end{theorem}

We will also use the following convenient version of the Saturation principle \cite{HVNSPrinciples}:
\begin{theorem}[Quantifier switching]\cite{HVNSPrinciples}
Let $A$ be an internal set. For each $n\in\N$, let $P_n(x)$, $Q_n(x)$ be internal formulas. If $P_n$ gets stronger as $n$ increases (i.e., for each $n\in\N$ and $x\in A$, $P_{n+1}(x)\implies P_n(x)$) and if
\[(\forall n,m\in\N) (\exists x\in A) (P_n(x) \,\& \, \neg Q_m(x)),
\]
then also
\[
(\exists x\in A) (\forall n\in\N) (P_n(x) \,\& \, \neg Q_n(x)).
\]
\end{theorem}

\section{Internal subsets of $\ster\R^d$ and internal functions on $\ster\R^d$}
If $A\subseteq\ster\R^d$, we call the exterior of $A$ (cf.\ also \cite{Giordano-Kunzinger-PEMS})
\[\ext(A):= \{x\in \ster\R^d: x\sterne a, \ \forall a\in A\} = \{x\in \ster\R^d: \abs{x-a}>0,\ \forall a\in A\}.\]
(Recall that $x > 0$ means $x \mathop{\ster >} 0$, i.e.\ $x_\eps > 0$ for small $\eps$.)\\
If $A$ is internal, then $\ext(A) = \{x\in \ster\R^d: x\sternotin A\}$. By transfer on
\[(\forall X\in\Powerset(\R^d)) (X\in\Powerset(\R^d)\setminus\{\R^d\} \iff (\exists y\in\R^d) (y\notin X)\bigr)\]
we find that $A\in\ster(\Powerset(\R^d)\setminus\{\R^d\})$ iff $A\subseteq\ster\R^d$ is internal with $\ext(A)\ne\emptyset$.

Let $\co(A):= \R^d \setminus A$. Considering $\co$: $\Powerset(\R^d)\setminus\{\R^d\}\to\Powerset(\R^d)\setminus\{\R^d\}$, by transfer on
\[(\forall X\in\Powerset(\R^d)\setminus\{\R^d\})(\forall y\in\R^d) (y\in\co (X) \iff y\notin X)\]
we find that $\ext(A)= \sterco(A)$ for each internal $A\subseteq\ster\R^d$ with $\ext(A)\ne\emptyset$.

If $R$ is a ring and $e\in R$ is idempotent (i.e., $e^2=e$), then we denote its complement idempotent $e_c:= 1-e$. We denote $R_\Idem:= \{e\in R: e^2=e\}$.

For any $A\subseteq\ster\R^d$, we denote its interleaved closure (cf.\ also \cite{HVInternal})
\[\interl(A):= \Bigl\{\sum_{j=1}^m a_j e_j: m\in\N, a_j\in A, e_j\in\ster\R_\Idem, \sum_{j=1}^m e_j=1\Bigr\}.\]
Again by transfer, an internal set $A\subseteq\ster\R^d$ is closed under interleaving, i.e.\ $A=\interl(A)$. For internal $A,B\subseteq\ster\R^d$,
\[A\stercup B = \interl(A\cup B) = \{xe+ye_c: x\in A, y\in B, e\in\ster\R_\Idem\}\]
is internal. For internal $A\subseteq\ster\R^d$ with $\ext(A)\ne\emptyset$, we have that $\interl(A\cup\ext(A))=A \stercup \sterco(A)=\ster\R^d$.

Since, by transfer, $\sterco(\{x\in\ster\R: x\le 0\}) = \{x\in\ster\R: x>0\}$, we have in particular (although the order on $\ster\R$ is not total)
\begin{equation}\label{eq-total-order-ersatz}
(\forall x\in\ster\R) (\exists e\in\ster\R_\Idem) (xe\le 0 \land x e_c > 0).
\end{equation}

An internal map $u$: $A\subseteq\ster\R^d\to\ster\R^{d'}$ has interleaved values, i.e.
\[
(\forall x,y\in A)(\forall e\in\ster\R_\Idem) (u(x e + y e_c) = u(x) e + u(y) e_c)
\]
or, equivalently,
\[(\forall x,y\in A)(\forall e\in\ster\R_\Idem) (xe=ye\implies u(x)e=u(y)e).\]

For $A\subseteq E\subseteq\R^d$, we also denote $\ext_E(A): = E\cap \ext(A)$, the exterior of $A$ in $E$.
\begin{lemma}\label{lemma-union-with-exterior}
Let $E\subseteq\ster\R^d$ be internal. Let $A$ be a countable union of internal subsets of $E$ and let $\ext_E(A)\ne \emptyset$. Let $P(x)$ be an internal formula. If $P(x)$ holds for each $x\in A\cup \ext_E(A)$, then $P(x)$ holds for each $x\in E$.
\end{lemma}
\begin{proof}
As $\{x\in E: P(x)\}$ is internal, $P(x)$ also holds for each $x\in \interl(A\cup \ext_E(A))$. Let $A=\bigcup_{n\in\N} A_n$ with $A_n\subseteq E$ internal. Let $n\in\N$.
If $x\in \ext_E(A)$, then $x\sternotin A_1$, \dots, $x\sternotin A_n$, and thus, by transfer, $x\sternotin A_1\stercup \cdots\stercup A_n$. Thus
\[(\exists C\in \ster\Powerset(\R^d)) \bigl((\exists x\in E) (x \sternotin C)\land A_1\stercup\cdots\stercup A_n\subseteq C\subseteq E \land (\forall x\in C) P(x)\bigr).\]
By Quantifier switching, we find an internal $C\subseteq E$ with $\ext_E(C)\ne\emptyset$, $A\subseteq C$ and $P(x)$, $\forall x\in C$. Then also $P(x)$, $\forall x\in \interl(C\cup\ext_E(A))\supseteq \interl(C\cup\ext_E(C)) = E$.
\end{proof}

\begin{corollary}\label{cor-zero-on-union}
Let $E\subseteq\ster\R^d$ be internal. Let $(r_n)_{n\in\N}$ be a monotone sequence in $\ster\R$ and $r\in\ster\R$. Let $u$ be an internal map $E\to\ster\R^{d'}$. Let
\[A:= \{x\in E: \abs x\le r_n,\text{ for some } n\in\N\}\quad \text{and}\quad B:=\{x\in E: \abs x> r_n, \forall n\in\N\}\]
or
\[A:=\{x\in E: \abs x\ge r_n, \text{ for some }n\in\N\}\quad \text{and}\quad B:= \{x\in E: \abs x< r_n,\forall n\in\N\}.\]
If $A, B\ne\emptyset$ and $\abs{u}\le r$ on $A\cup B$, then $\abs{u}\le r$ on $E$.
\end{corollary}
\begin{proof}
Consider the first case (the second case is similar). Let $A_n:= \{x\in\ster\R^d: \abs x\le r_n\}$ and $B_n:= \{x\in\ster\R^d: \abs x> r_n\}$. As $(r_n)_n$ is monotone and $A\ne\emptyset$, w.l.o.g.\ $A_n\cap E\ne \emptyset$, and thus $A_n\cap E$ are internal. In order to apply the previous lemma with $P(x):=(\abs {u(x)}\le r)$, we show that $B=\ext_E(A)$. By transfer, $\ext(A_n)=\sterco(A_n)= B_n$. The result follows, since $\ext(\bigcup_i X_i) = \bigcap_i \ext(X_i)=\{x\in\ster\R^d: |x-y|>0, \forall i,\forall x\in X_i\}$ for any $X_i\subseteq\ster\R^d$. 
\end{proof}

In practice, we will often use the following convenient version of over- and underspill:
\begin{proposition}[Overspill]
Let $P(m)$ be an internal formula. Then
\[P(m) \text{ holds for sufficiently small }m\in\ster\N_\infty \iff P(m) \text{ holds for sufficiently large }m\in\N\]
i.e., $(\exists M_0\in\ster\N_\infty) (\forall m\in\ster\N_\infty, m\le M_0) P(m)$ $\iff$ $(\exists m_0\in\N) (\forall m\in\N, m\ge m_0) P(m)$.
\end{proposition}
\begin{proof}
Let $A:= \{m\in\ster\N: P(m)\}$.

\item $\Rightarrow$: by assumption and eq.\ \eqref{eq-total-order-ersatz},  $\ster\N_\infty\subseteq A\stercup \{m\in\ster\N: m\ge M_0\}$. Then also
\[\ster\N_\infty\subseteq B:= \bigl\{n\in\ster\N: (\forall m\in\ster\N, m\ge n)(m\in A\stercup \{m\in\ster\N: m\ge M_0\})\bigr\}.\]
By underspill, $B\cap \N\ne\emptyset$. In particular, each sufficiently large $m\in\N$ belongs to $\N\cap (A\stercup \{m\in\ster\N: m\ge M_0\}) = \N\cap A$.

\item $\Leftarrow$: by assumption, $\N\subseteq A \stercup \{m\in\ster\N: m\le m_0\}$. By overspill, each sufficiently small $m\in\ster\N_\infty$ belongs to $\ster\N_\infty\cap (A\stercup \{m\in\ster\N: m\le m_0\})= \ster\N_\infty\cap A$.
\end{proof}

\section{Moderateness and $\Mod^\infty$-regularity}
We denote the infinitesimal $\rho := [\eps] >0$.
We call $x\in\ster\R^d$ \defstyle{moderate} (notation: $x\in\ster\R^d_{\Mod}$) if $|x|\le \rho^{-N}$ for some $N\in\N$, and \defstyle{negligible} (notation: $x\approxeq 0$) if $|x|\le \rho^n$ for each $n\in\N$.\\
In this paper, $\Omega\subseteq\R^d$ will denote an open set. We call $\ster\Omega_c:= \bigcup_{K\csub\Omega} \ster K$. We denote $\mathcal K(\Omega):= \{K\subseteq \Omega: K$ is compact, $K\ne\emptyset\}$.

Let $K\in\ster(\mathcal K(\R^d))$. Let $u$ be a map $\ster\R^d\to\ster\C$. We say that $u$ is supported in $K$ (cf.\ also \cite{Giordano-Kunzinger-PEMS}; notation: $\supp(u)\subseteq K$) if $u=0$ on $\ext (K)$. For $u\in\ster(\Cnt[\infty](\R^d))$, this means (by transfer) that $\ster\supp(u)\subseteq K$ (i.e., $\supp(u_\eps)\subseteq K_\eps$ for small $\eps$). 
We say that $u$ is compactly supported in $A\subseteq\ster\R^d$ if $\supp(u)\subseteq K$ for some $K\in\ster{(\mathcal K(\R^d))}$ with $K\subseteq A$. We define
\begin{align*}
\Mod(A) & := \{u\in\ster(\Cnt[\infty](\R^d)): \partial^\alpha u(x)\in\ster\C_\Mod, \forall \alpha\in\N^d,\forall x\in A\}\\
\Mod^\infty(A)& := \{u\in\ster(\Cnt[\infty](\R^d)): (\forall x\in A) (\exists N\in\N) (\forall \alpha\in\N^d) |\partial^\alpha u(x)|\le \caninf^{-N}\}\\
\Mod_c(A)& := \{u\in\Mod(\ster\R^d): u \text{ is compactly supported in } A\}\\
\Mod^\infty_c(A)& := \Mod^\infty(A)\cap \Mod_c(A).
\end{align*}

\begin{lemma}\label{lemma-internal-subset-of-union}
Let $A=\bigcup_{n\in\N} A_n$ with $A_n\subseteq\ster\R^d$ internal. If $B\subseteq A$ is internal, then $B\subseteq A_n$ for some $n\in\N$.
\end{lemma}
\begin{proof}
Seeking a contradiction, suppose that $(\forall n\in\N)$ $(\exists x\in B)$ $(x\notin A_n)$. By Quantifier switching, there would exist $x\in B$ such that $x\notin A$.
\end{proof}

\begin{proposition}\label{compactly-supported-in-union}
Let $A=\bigcup_{n\in\N} A_n$ with $A_n\subseteq\ster\R^d$ internal. Let $u$ be a map $\ster\R^d\to\ster\C$. Then $u$ is compactly supported in $A$ iff $u$ is compactly supported in $A_n$ for some $n\in\N$.
\end{proposition}
\begin{proof}
If $K\in\ster(\mathcal K(\R^d))$ and $K\subseteq A$, then $K\subseteq A_n$ for some $n$ by the previous lemma.
\end{proof}

\begin{corollary}
\[\Mod_c(\ster\Omega_c)= \{u\in\Mod(\ster\R^d): (\exists K\csub\Omega) (\supp(u)\subseteq\ster K)\}.\]
\end{corollary}

\begin{proposition}\label{M-infty-kar}
Let $A=\bigcup_{n\in\N} A_n$ with $A_n\subseteq\ster\R^d$ internal. Let $u\in\ster(\Cnt[\infty](\R^d))$. Then
\[u\in \Mod^\infty(B)\text{ for some internal }B\supseteq A\ \text{ iff }\ (\exists N\in\N) (\forall \alpha\in\N^d) (\forall x\in A) |\partial^\alpha u(x)|\le \caninf^{-N}.\]
\end{proposition}
\begin{proof}
$\Rightarrow$: by Quantifier switching, as in \cite[Prop.\ 7.6]{HV-NSPrinciples-tutorial}.

\item $\Leftarrow$: for each $n\in\N$,
\[(\exists B\in\ster\Powerset(\R^d)) \bigl(A_1\stercup\cdots\stercup A_n\subseteq B \land (\forall x\in B)(\forall \alpha\in\ster\N^d, \abs\alpha\le n) |\partial^\alpha u(x)|\le\caninf^{-N}\bigr)\]
since $\{(\alpha,x)\in\ster\N^d\times\ster\R^d: \abs{\partial^\alpha u(x)}\le \caninf^{-N}\}$ is closed under interleaving and $\{\alpha\in\ster\N^d: \abs\alpha\le n\}$ is, by transfer, the interleaved closure of the finite set $\{\alpha\in\N^d: \abs\alpha\le n\}$. The result then follows by Quantifier switching.
\end{proof}
\begin{corollary}
\[\Mod^\infty(\ster\Omega_c)=\{u\in\ster(\Cnt[\infty](\R^d)): (\forall K\csub\Omega) (\exists N\in\N) (\forall \alpha\in\N^d) (\forall x\in \ster K) |\partial^\alpha u(x)|\le \caninf^{-N}\}.\]
\end{corollary}

Thus (cf.\ also \cite{HV-NSPrinciples-tutorial}) the Colombeau algebras $\Gen(\Omega)$, $\Gen^\infty(\Omega)$, $\Gen_c(\Omega)$ and $\Gen^\infty_c(\Omega)$ are quotients of $\Mod(\ster\Omega_c)$, $\Mod^\infty(\ster\Omega_c)$, $\Mod_c(\ster\Omega_c)$ and $\Mod_c^\infty(\ster\Omega_c)$, respectively, modulo
\[\Null(\ster\Omega_c):= \{u\in\ster(\Cnt[\infty](\R^d)): \partial^\alpha u(x)\approxeq 0, \forall \alpha\in\N^d,\forall x\in\ster\Omega_c\}.\]

We call $x\in\ster\R^d$ \defstyle{fast scale} if $x$ belongs to
\[\ster\R^d_{fs}:= \{x\in\ster\R^d: (\exists a\in\R_{>0}) (\abs x\ge \caninf^{-a})\}\]
and we call $x$ \defstyle{slow scale} if $x$ belongs to
\[\ster\R^d_{ss}:=\{x\in\ster\R^d: (\forall a\in\R_{>0}) (\abs x\le \caninf^{-a})\}.\]
We call $x\in\ster\R^d$ a \defstyle{slow scale infinitesimal} (notation: $x\approx_\slow 0$) if $x\approx 0$ and $\frac1{\abs{x}}$ is slow scale, i.e., if
\[\caninf^a\le \abs{x}\le a,\quad \forall a\in\R_{>0}\]
and we call $x$ a \defstyle{fast scale infinitesimal} (notation: $x\approx_\fast 0$) if
\[\abs{x}\le \caninf^a,\quad \text{for some } a\in\R_{>0}.\]
We write $x\approx_\fast y$ (resp.\ $x\approx_\slow y$) for $x-y\approx_\fast 0$ (resp.\ $x-y\approx_\slow 0$).\\
We call a \defstyle{slow scale neighbourhood} of $x_0\in\ster\R^d$ any set that contains $\{x\in\ster\R^d: \abs{x-x_0}\le r\}$ for some $r\approx_\slow 0$ ($r\in\ster\R_{>0}$). A \defstyle{conic slow scale neighbourhood} of $\xi_0\in \ster\Sphere$ is a cone $\Gamma\subseteq\ster\R^d$ with vertex $0$ that contains a slow scale neighbourhood of $\xi_0$ (thus there exists some $r\approx_\slow 0$ ($r\in\ster\R_{>0}$) such that $\abs[\big]{\frac{\xi}{\abs\xi}-\xi_0}\le r\implies \xi \in \Gamma$).

By Cor.\ \ref{cor-zero-on-union}, we obtain:
\begin{lemma}\label{lemma-zero-on-fs-and-ss}
Let $u\in\ster(\Cnt[\infty](\R^d))$. If $u\approxeq 0$ on $\ster\R^d_{fs}\cup \ster\R^d_{ss}$, then $u\approxeq 0$ on $\ster\R^d$.
\end{lemma}

\section{$\Mod^\infty$-microlocal regularity}
\begin{definition}
$\Mod_{\Schwartz}(\ster\R^d)=\{u\in\ster(\Schwartz(\R^d)): x^\alpha\partial^\beta u(x)\in \ster\C_\Mod,\ \forall x\in\ster\R^d$, $\forall \alpha$, $\beta$ $\in\N^d\}$.
\end{definition}

To keep this paper self-contained, we recast some properties concerning $\Mod_\Schwartz$ and the Fourier transform in this setting (cf.\ also \cite{HVPointwise}):
\begin{lemma}\leavevmode\label{lemma-fourier}
\begin{enumerate}
\item $\Mod_c(\ster\R^d_\Mod)\subseteq\Mod_\Schwartz(\ster\R^d)$.
\item The Fourier transform $\Fourier$ is a bijection $\Mod_\Schwartz(\ster\R^d)\to \Mod_\Schwartz(\ster\R^d)$.
\item Let $u\in\Mod_\Schwartz(\ster\R^d)$. Then for each $k\in\N$, there exists $m\in\N$ s.t.\ $\int_{|x|\ge\caninf^{-m}}\abs{u}\le \caninf^k$.
\item Let $u\in\Mod_\Schwartz(\ster\R^d)$ and $u(x)\approxeq 0$ for each $x\in\ster\R^d_\Mod$. Then $\int \abs u\approxeq 0$.
\item Let $\phi\in\Mod^\infty_c(\ster\R^d_\Mod)$. Then $\fourier \phi(\xi)\approxeq 0$ for all $\xi\in\ster\R^d_{fs}$.
\item Let $u\in\Mod_\Schwartz(\ster\R^d)$. If $u(x)\approxeq 0$ for all $x\in\ster\R^d_{fs}$, then $\fourier u\in\Mod^\infty(\ster\R^d)$.
\end{enumerate}
\end{lemma}
\begin{proof}
1. By definition, $\Mod_c(\ster\R^d_\Mod)\subseteq\ster(\Cnt[\infty]_c(\R^d))\subseteq\ster(\Schwartz(\R^d))$. Let $u\in\Mod_c(\ster\R^d_\Mod)$. By Prop.\ \ref{compactly-supported-in-union}, $\supp(u)\subseteq B(0,\caninf^{-M})$, for some $M\in\N$. Thus if $\abs x >\caninf^{-M}$, $x^\alpha\partial^\beta u(x)=0$. If $\abs x\le \caninf^{-M}$, $\abs{x^\alpha\partial^\beta u(x)}\le \caninf^{-M\abs\alpha}\abs{\partial^\beta u(x)}\in\ster\R_\Mod$. The result follows by Cor.\ \ref{cor-zero-on-union}.

\item 2. Let $u\in\Mod_\Schwartz(\ster\R^d)$. As $\Fourier$: $\Schwartz\to\Schwartz$ is continuous, there exist $C\in\R$ and $N\in\N$ such that for each $\alpha,\beta\in\N^d$ (by transfer) $\sup_{\xi\in\ster\R^d}\abs{\xi^\alpha\partial^\beta\fourier u(\xi)}\le C \sup_{x\in\ster\R^d,\abs{\alpha'},\abs{\beta'}\le N}\abs[\big]{x^{\alpha'}\partial^{\beta'} u(x)}$ $\in\ster\R_\Mod$. Hence $\fourier u\in\Mod_\Schwartz(\ster\R^d)$. The result follows by Fourier inversion.

\item 3. If $m\in\ster\N_\infty$, then $\int_{\abs x\ge \caninf^{-m}} \abs u\le C\int_{\abs x\ge \caninf^{-m}}\langle x\rangle^{-d-2}\,dx\le C\caninf^{m}\int_{\ster\R}\langle x\rangle^{-d-1}\,dx\le\caninf^k$ ($C\in\ster\R_\Mod$). The result follows by overspill.

\item 4. Let $k\in\N$ and $m\in\N$ as in part (3). Then
\[\textstyle \int \abs{u} = \int_{\abs x\ge \caninf^{-m}} \abs{u} +\int_{\abs x\le \caninf^{-m}} \abs u\le \caninf^k+ \sup_{\abs x\le \caninf^{-m}}\abs{u} \cdot \int_{\abs x\le \caninf^{-m}}1\approxeq \caninf^k.\]

\item 5. By Prop.\ \ref{compactly-supported-in-union}, $\supp(\phi)\subseteq B(0,\caninf^{-M})$ for some $M\in\N$. Let $|\xi|\ge \caninf^{-1/N}$ ($N\in\N$). Then for each $\alpha\in\N^d$, $\abs[\big]{\xi^\alpha \fourier \phi(\xi)}\le \int_{\abs x\le \caninf^{-M}} \abs{\partial^\alpha \phi}\le C \in \ster\R_\Mod$ ($C$ is independent of $\alpha$). Then for any $m\in\N$, $\abs[]{\fourier \phi(\xi)}\le C\caninf^{-1}\abs\xi^{-m}\le C\caninf^{m/N-1}$.

\item 6. By overspill, there exists $k\in\ster\N_\infty$ such that $\abs{u(x)}\le \caninf^{k}$ for each $x\in\ster\R^d$ with $\abs x\ge \caninf^{-1/k}$. Let $\alpha\in\N^d$. For a suitable $m\in\N$, $\int_{\abs x\ge \caninf^{-m}}\abs {x^\alpha u(x)}\,dx \le 1$ by part (3). Further, $\int_{\caninf^{-1/k}\le \abs x\le \caninf^{-m}}\abs {x^\alpha u(x)}\,dx\approxeq 0$ and
\[\textstyle\int_{\abs x\le \caninf^{-1/k}}\abs {x^\alpha u(x)}\,dx\le \caninf^{\abs{\alpha}/k}\sup_{\abs x\le \caninf^{-1/k}}\abs{u(x)} \le \caninf^{-1} \sup_{\abs x\le \caninf^{-1/k}}\abs{u(x)}.\]
Thus for each $\xi\in\ster\R^d$, $\abs{\partial^\alpha \fourier u(\xi)}\le \int \abs{x^\alpha u(x)}\,dx\le C\in\ster\R_\Mod$ ($C$ is independent of $\alpha$).
\end{proof}

We denote $\Sphere := \{x\in\R^d: \abs x = 1\}$.
\begin{definition}\label{df-regular}
$u\in \Mod(\ster\Omega_c)$ is $\Mod^\infty$-microlocally regular at $(x_0,\xi_0)\in\ster \Omega_c\times \ster\Sphere$ if there exists $v\in\Mod_c(\ster\Omega_c)$ such that
\[u(x)=v(x), \ \forall x\approx_\fast x_0 \qquad \text{and}\qquad \fourier v(\xi)\approxeq 0,\ \forall \xi \in \ster\R^d_{fs} \text{ with }  \frac{\xi}{\abs\xi}\approx_\fast \xi_0.\]
\end{definition}

\begin{proposition}\label{cut-off}
Let $v\in\Mod_c(\ster\Omega_c)$, $\phi\in\Mod^\infty(\ster\Omega_c)$ and 
$\xi_0\in\ster\Sphere$. Let $\fourier v(\xi)\approxeq 0$ for each $\xi\in\ster\R^d_{fs}$ with $\frac{\xi}{\abs\xi}\approx_\fast \xi_0$. Then also $\fourier{\phi v}(\xi)\approxeq 0$ for each $\xi\in\ster\R^d_{fs}$ with $\frac{\xi}{\abs\xi}\approx_\fast \xi_0$.
\end{proposition}
\begin{proof}
W.l.o.g., $\phi\in\Mod^\infty_c(\ster\Omega_c)$.
Fix $\xi\in \ster\R^d_{fs}$ with $\frac{\xi}{\abs\xi}\approx_\fast \xi_0$. Then
\[\fourier{\phi v}(\xi) = \int\fourier\phi(\xi-\eta)\fourier v(\eta)\,d\eta.\]
By Lemma \ref{lemma-fourier}, $\fourier\phi(\xi-\eta)\approxeq 0$ if $\xi-\eta\in\ster\R^d_{fs}$.

Now let $\xi-\eta\in\ster\R^d_{ss}$. Then also $\eta\in\ster\R^d_{fs}$. Since $\frac{\abs{\xi-\eta}}{\abs\xi}\approx_\fast 0$,
\[
\frac{\eta}{\abs{\eta}}\approx_\fast \frac{\eta}{\abs\xi}\approx_\fast \frac{\xi}{\abs\xi}\approx_\fast \xi_0.
\]
Hence $\fourier v(\eta)\approxeq 0$. By Lemma \ref{lemma-zero-on-fs-and-ss}, $\fourier\phi(\xi-\eta)\fourier v(\eta)\approxeq 0$ for each $\eta\in\ster\R^d$. As $\phi v \in\Mod_\Schwartz$, also $\fourier{\phi v}\in\Mod_\Schwartz$. By Lemma \ref{lemma-fourier}, $\fourier{\phi v}(\xi)\approxeq 0$.
\end{proof}

\begin{corollary}\label{regularity-via-cut-off}
Let $\phi\in\Mod^\infty(\ster\Omega_c)$. If $u\in\Mod(\ster\Omega_c)$ is $\Mod^\infty$-microlocally regular at $(x_0,\xi_0)$, then also $\phi u$ is $\Mod^\infty$-microlocally regular at $(x_0,\xi_0)$.
\end{corollary}

We use the following notation. We fix $\phi_0\in\test(B(0,1))$ with $0\le\phi_0\le 1$ and with $\phi_0(x)=1$ for each $x\in B(0,1/2)$. For $m\in\ster\N$ and $x_0\in\ster\R^d$, we denote
\[\phi_{m,x_0}(x):= \phi_0\Bigl(\frac{x-x_{0}}{\caninf^{1/m}}\Bigr).\]

\begin{proposition}\label{def-regular-with-cutoff}
For $u\in \Mod(\ster\Omega_c)$ and $(x_0,\xi_0)\in\ster \Omega_c\times \ster\Sphere$, the following are equivalent:
\begin{enumerate}
\item $u$ is $\Mod^\infty$-microlocally regular at $(x_0,\xi_0)$
\item there exists $\phi\in\Mod_c^\infty(\ster\Omega_c)$ such that
\begin{equation}\label{eq-def-M-infty-mr}
\phi(x)=1, \ \forall x\approx_\fast x_0 \qquad \text{and}\qquad \fourier{\phi u}(\xi)\approxeq 0,\ \forall \xi \in \ster\R^d_{fs}\text{ with }\frac\xi{\abs\xi}\approx_\fast \xi_0
\end{equation}
\item there exists $\phi\in\Mod_c^\infty(\ster\Omega_c)$ and $R\in\ster\R_{ss}$ such that
\[\begin{cases}
\abs{\partial^\alpha\phi(x)}\le R\\
\abs{\phi(x)}\ge \frac1R,
\end{cases}
\!\forall x\approx_\fast x_0, \ \forall \alpha\in\N^d \text{ and }\ \fourier{\phi u}(\xi)\approxeq 0,\ \forall \xi \in \ster\R^d_{fs}\text{ with }\frac\xi{\abs\xi}\approx_\fast \xi_0\]
\item there exists $m\in\ster\N_\infty$ (with $m$ sufficiently small, such that $\phi_{m,x_0}\in\Mod_c(\ster\Omega_c)$) such that $\fourier{\phi_{m,x_0} u}(\xi)\approxeq 0,\ \forall \xi \in \ster\R^d_{fs}\text{ with }\frac\xi{\abs\xi}\approx_\fast \xi_0$.
\end{enumerate}
\end{proposition}
\begin{proof}
$(1)\Rightarrow(2)$: choose $v$ as in the definition of $\Mod^\infty$-microlocal regularity. By overspill, there exists some $m\in\ster\N_\infty$ such that $u(x)=v(x)$ for each $x\in\ster\R^d$ with $\abs{x-x_0}\le \caninf^{1/m}$. For $\phi:=\phi_{m,x_0}\in\Mod_c^\infty(\ster\Omega_c)$, we have $\phi(x)=1$ for each $x\approx_\fast x_0$ and $\phi u=\phi v$. By Proposition \ref{cut-off}, $\fourier{\phi u}(\xi)=\fourier{\phi v}(\xi)\approxeq 0$ for each $\xi\in\ster\R^d_{fs}$ with $\frac\xi{\abs\xi}\approx_\fast\xi_0$.

$(2)\implies(3)$: trivial.

$(3)\implies(4)$: by overspill, there exists $m\in\ster\N_\infty$ such that $\abs{\phi(x)}\ge 1/R$ and $\abs{\partial^\alpha\phi(x)}\le R$ for each $x\in\ster\R^d$ with $\abs{x-x_0}\le \caninf^{1/m}$ and for each $\alpha\in\ster\N^d$ with $\abs\alpha\le m$. Then $\phi_{m,x_0}\in\Mod_c^\infty(\ster\Omega_c)$ (if $m\in\ster\N_\infty$ is sufficiently small) and $\frac{\phi_{m,x_0}}\phi\in\Mod^\infty_c(\ster\Omega_c)$, whence $\fourier{\phi_{m,x_0} u}(\xi)=\Fourier\bigl(\frac{\phi_{m,x_0}}{\phi}\phi u\bigr)(\xi)\approxeq 0$ for each $\xi \in \ster\R^d_{fs}$ with $\frac\xi{\abs\xi}\approx_\fast \xi_0$ by Prop.\ \ref{cut-off}.

$(4)\implies (1)$: let $v:=\phi_{m,x_0} u\in \Mod_c(\ster\Omega_c)$.
\end{proof}

\section{Consistency with $\Mod^\infty$-regularity}
We now proceed to show that the projection of the wave front set in the first coordinate is the singular support (Theorem \ref{proj-of-WF-is-singular-support}). 

\begin{lemma}
Let $u\in\Mod(\ster\Omega_c)$ and $V\subseteq\ster\Omega_c$ be internal. Then there exists $m_0\in\ster\N_\infty$ such that $\phi_{m,x}\in\Mod_c(\ster\Omega_c)$ for each $x\in V$ and each $m\le m_0$ ($m\in\ster\N$).
\end{lemma}
\begin{proof}
By Lemma \ref{lemma-internal-subset-of-union}, $V\subseteq\ster K$ for some compact $K\subset \Omega$. Thus we can choose $m_0\in\ster\N_\infty$ s.t.\ $\caninf^{1/m_0}\le \frac12 d(K,\R^d\setminus\Omega)\in\R_{>0}$.
\end{proof}

We first prove the following uniform version of Proposition \ref{def-regular-with-cutoff}:
\begin{theorem}\label{thm-projection-of-wave-front-2}
Let $V\subseteq\ster\Omega_c$ be internal and let $\Gamma\subseteq \ster\R^d$ be an internal cone. Let $u \in \Mod(\ster\Omega_c)$ be $\Mod^\infty$-microlocally regular at $(x_0,\xi_0)$, for each $x_0\in V$ and each $\xi_0\in\ster\Sphere\cap\Gamma$. Then there exists $k\in\ster\N_\infty$ such that
$\fourier{\phi_{k,x} u}(\xi)\approxeq 0$ for each $x\in V$ and $\xi\in\Gamma\cap\ster\R^d_{fs}$.
\end{theorem}
\begin{proof}
Let $m_0\in\ster\N_\infty$ as in the previous lemma.

(1) For each $k\in\N$,  $x_0\in V$ and $\xi_0\in\ster\Sphere\cap \Gamma$, there exists $m\in\ster\N$ with $k<m\le m_0$ s.t.
\[\abs[]{\fourier{\phi_{m,x_0} u}(\xi)}\le \rho^{k},\quad \forall \xi\in\ster\R^d, \abs{\frac\xi{\abs\xi}-\xi_0}\le\rho^{1/k},\, \abs\xi\ge\rho^{-1/k}\]
since every $m\in\ster\N_\infty$ as in Prop.\ \ref{def-regular-with-cutoff}(4) satisfies this condition ($m$ depends on $x_0$, $\xi_0$).

\item (2) For each $k\in\N$, $\phi_{k,0}\phi_{m,0}=\phi_{k,0}$,  for each $m\in\ster\N$ with $m>k$.

\item By overspill, (1) and (2) simultaneously hold for some $k\in\ster\N_\infty$ ($k$ does not depend on $x_0$, $\xi_0$). Then in particular $\fourier{\phi_{m,x_0} u}(\xi)\approxeq 0$, for each $\xi\in\ster\R^d_{fs}$ with $\frac\xi{\abs\xi}\approx_\fast\xi_0$. Since $\phi_{k,x_0}\in\Mod^\infty(\ster\Omega_c)$, Prop.\ \ref{cut-off} shows that also $\fourier{\phi_{k,x_0} u}(\xi)=\Fourier\bigl(\phi_{k,x_0}(\phi_{m,x_0} u)\bigr)(\xi)\approxeq 0$ for each $\xi\in\ster\R^d_{fs}$ with $\frac\xi{\abs\xi}\approx_\fast\xi_0$. As $x_0\in V$ and $\xi_0\in \ster \Sphere\cap \Gamma$ are arbitrary, the result follows.
\end{proof}
\begin{remark}
The previous proof indicates the need to go beyond the ring $\GenR$ of generalized Colombeau numbers. Although one can also formulate an overspill principle in this context \cite{HVInternal}, one cannot distinguish between $\caninf^k$ ($k$ infinitely large) and $0$ in $\GenR$.
\end{remark}

\begin{theorem}\label{proj-of-WF-is-singular-support}
Let $x_0\in\ster\Omega_c$. For $u\in\Mod(\ster\Omega_c)$, the following are equivalent:
\begin{enumerate}
\item $u$ is $\Mod^\infty$-microlocally regular at $(x_0,\xi_0)$, for each $\xi_0\in\ster\Sphere$
\item $u\in\Mod^\infty(V)$ for some slow scale neighbourhood $V$ of $x_0$.
\end{enumerate}
\end{theorem}
\begin{proof}
$(1)\Rightarrow(2)$: by Theorem \ref{thm-projection-of-wave-front-2} (with $V:=\{x_0\}$ and $\Gamma:= \ster\R^d$), we find $k\in\ster\N_\infty$ such that $\fourier{\phi_{k,x_0} u}(\xi)\approxeq 0$ for each $\xi\in\ster\R^d_{fs}$, and $\phi_{k,x_0}\in\Mod_c(\ster\Omega_c)$. Hence $\phi_{k,x_0} u\in\Mod_c(\ster\Omega_c)\subseteq\Mod_\Schwartz(\ster\R^d)$. By Lemma \ref{lemma-fourier}, $\phi_{k,x_0}u\in\Mod^\infty(\ster\R^d)$. As $\phi_{k,x_0}=1$ on a slow scale neighbourhood $V$ of $x_0$, also $u\in \Mod^\infty(V)$.

$(2)\Rightarrow(1)$: there exists $\phi\in\Mod_c^\infty(\ster\Omega_c)$ with $\phi = 1$ on a slow scale neighbourhood of $x_0$ and with $\phi u\in\Mod_c^\infty(\ster\Omega_c)$ (e.g., $\phi=\phi_{m,x_0}$ for a sufficiently small $m\in\ster\N_\infty$). By Lemma \ref{lemma-fourier}, $\fourier{\phi u}(\xi) \approxeq 0$ for each $\xi\in\ster\R^d_{fs}$.
\end{proof}

We can equivalently reformulate the condition in the previous Theorem:
\begin{proposition}
For $u\in\Mod(\ster\Omega_c)$, the following are equivalent:
\begin{enumerate}
\item $u\in\Mod^\infty(V)$ for some slow scale neighbourhood $V$ of $x_0$
\item $(\exists N\in\N)$ $(\forall \alpha\in\N^d)$ $(\forall x\in\ster\R^d, x\approx_\fast x_0)$ $(\abs{\partial^\alpha u(x)}\le\caninf^{-N})$
\end{enumerate}
\end{proposition}
\begin{proof}
For $n\in\ster\N$, let $A_n:=B(x_0,\caninf^{1/n})$. Then $A_n$ is internal and $\{x\in\ster\R^d: x\approx_\fast x_0\}=\bigcup_{n\in\N} A_n$. Each slow scale neighbourhood of $x_0$ and each internal set containing $\{x\in\ster\R^d: x\approx_\fast x_0\}$ contains $A_m$ for some $m\in\ster\N_\infty$ by overspill. Thus the result follows by Prop.\ \ref{M-infty-kar}.
\end{proof}

\section{Connection with $\Gen^\infty$-microlocal regularity}
\begin{definition}
We denote $\mathcal E(\R^d):= \{A\subseteq\R^d: A$ is finite, $A\ne \emptyset\}$. Elements of $\ster(\mathcal E(\R^d))$ are called hyperfinite subsets of $\ster\R^d$. Considering the number of elements as a map $\#$: $\mathcal E(\R^d)\to \N$, we call $\ster\#(A)\in\ster\N$ the number of elements of $A\in\ster(\mathcal E(\R^d))$.

Similarly, we can extend other operations to hyperfinite sets, e.g., for $A\in\ster(\mathcal E(\R^d))$ and $u$: $\ster\R^d\to\ster\C$ internal, $\ster\sum_{x\in A} u(x)\in\ster\C$, where we consider $\sum$: $\mathcal E(\R^d)\times \{u: \R^d\to\C\}\to \C$. As the sum is one of the most basic operations, we will write $\sum:=\ster\sum$. The usual calculation rules hold by transfer. (For $A=[A_\eps]$, $A$ is hyperfinite iff $A_\eps$ is finite for small $\eps$, and $\# A = [\# A_\eps]$. For $u=[u_\eps]$, $\sum_{x\in A} u(x) = [\sum_{x\in A_\eps} u_\eps(x)]$, \dots)
\end{definition}

\begin{theorem}\label{consistency}
Let $u\in\Mod(\ster\Omega_c)$, $x_0\in\Omega$ and $\xi_0\in\Sphere$. Then the following are equivalent:
\begin{enumerate}
\item there exists $r\in\R_{>0}$ such that $u$ is $\Mod^\infty$-microlocally regular at $(x,\xi)$ for each $x\in B_{\ster\R^d}(x_0, r)$ and each $\xi\in\ster \Sphere$ with $\abs{\xi-\xi_0}\le r$
\item there exist $\phi\in\Mod_c^\infty(\ster\Omega_c)$, $r\in\R_{>0}$ and $R\in\ster\R_{ss}$ such that $\abs{\partial^\alpha\phi(x)}\le R$ and $\abs{\phi(x)}\ge 1/R$ for each $\alpha\in\N^d$ and each $x\in B_{\ster\R^d}(x_0, r)$, and
\[
\fourier{\phi u}(\xi)\approxeq 0,\quad \forall \xi\in\ster\R^d_{fs}, \abs{\frac{\xi}{\abs\xi}-\xi_0}\le r
\]
\item there exist $\psi\in\test(\Omega)$ with $\psi(x_0)=1$ and a conic neighbourhood $\Gamma\subseteq\R^d$ of $\xi_0$ s.t.
\[
\fourier{\psi u}(\xi)\approxeq 0,\quad \forall \xi\in\ster\Gamma\cap\ster\R^d_{fs}.
\] 
\end{enumerate}
\end{theorem}
\begin{proof}
$(1)\Rightarrow (2)$: Let $V:= B_{\ster\R^d}(x_0,r)$ and $\Gamma:=\{\xi\in\ster\R^d: \abs[\big]{\xi - \abs\xi \xi_0}\le \abs\xi r\}$. By Theorem \ref{thm-projection-of-wave-front-2}, we find $k\in\ster\N_\infty$ such that $\fourier{\phi_{k,x} u}(\xi)\approxeq 0$ for each $x\in V$ and $\xi\in\Gamma\cap\ster\R^d_{fs}$. For convenience, we use the norm $\norm{x}_\infty:= \max\{x_1,\dots,x_d\}$ on $\R^d$ (and its extension to a map $\ster\R^d\to\ster\R$). Consider a grid
\[G:=\Bigl\{x\in\ster\R^d: \norm{x-x_0}_\infty< \frac{r}{\sqrt d}, \ x\in \frac{\caninf^{1/k}}{\sqrt d}\ster\Z^d\Bigr\}.\]
Then $G$ is a hyperfinite set with at most $(2r\caninf^{-1/k}+1)^d\le \caninf^{-(d+1)/k}$ elements. Let $\psi:= \sum_{y\in G}\phi_{k,y}$. Let $\xi\in \Gamma\cap\ster\R^d_{fs}$. As $G\subseteq V$, $\abs[\big]{\fourier{\psi u}(\xi)}\le \sum_{y\in G} \abs[]{\fourier{\phi_{k,y}u}(\xi)}\le \caninf^{n-(d+1)/k}$ for each $n\in\N$. Hence $\fourier{\psi u}(\xi)\approxeq 0$.

Now let $x\in W:=\bigl\{x\in\ster\R^d:\norm{x-x_0}_\infty\le \frac r{2\sqrt d}\bigr\}$ arbitrary. Then there exists $y_0\in G$ such that $\norm{x-y_0}_\infty \le \frac{\caninf^{1/k}}{2\sqrt d}$, hence $\psi(x)\ge \phi_{k,y_0}(x) = 1$. On the other hand, let $x\in\ster\Omega_c$. Then there is at most a finite number $C_d\in\N$ (independent of $x$) of elements $y\in G$ for which $\abs{x-y}\le \caninf^{1/k}$. Hence for each $\alpha\in\N^d$,
\[\abs{\partial^\alpha \psi(x)}\le C_d \,\sup\limits_{y\in V} \abs{\partial^\alpha \phi_{k,y}(x)}\le C_{d,\alpha} \caninf^{-\abs{\alpha}/k}\le \caninf^{-1/\sqrt k} =:R\]
($C_{d,\alpha}\in\R$). In particular, $\psi\in\Mod_c^\infty(\ster\Omega_c)$.

$(2)\Rightarrow(3)$: as in Prop.\ \ref{def-regular-with-cutoff}, we can find $\psi\in\test(B(x_0,r))$ with $\psi(x_0)=1$ for which $\frac\psi\phi\in\Mod^\infty_c(\ster\Omega_c)$, and thus $\fourier{\psi u}(\xi)\approxeq 0$ for each $\xi\in\ster{\Gamma}\cap\ster\R^d_{fs}$ (for some conic neighbourhood $\Gamma\subseteq\R^d$ of $\xi_0$) by Prop.\ \ref{cut-off}.

$(3)\Rightarrow(1)$: by Prop.\ \ref{def-regular-with-cutoff}, as $\test(\Omega)\subseteq \Mod^\infty_c(\ster\Omega_c)$.
\end{proof}

The following lemma makes the connection with $\Gen^\infty$- and $\widetilde\Gen^\infty$-microlocal analysis of functions in $\Gen(\Omega)$ \cite{HVMicrolocal}.

For $x\in\ster\R^d$, we denote by $\widetilde x\in\GenR^d$ the equivalence class of $x$ modulo $\approxeq$. Similarly, for $u\in\Mod(\ster\Omega_c)$, we denote by $\widetilde u\in\Gen(\Omega )$ the equivalence class of $u$ modulo $\Null(\ster\Omega_c)$.
\begin{lemma}
Let $u\in\Mod(\ster\Omega_c)$, $x_0\in\ster\Omega_c$ and $\xi_0\in\ster\Sphere$. Then $u$ is $\Mod^\infty$-microlocally regular at $(x_0,\xi_0)$ iff $\widetilde u$ is $\widetilde\Gen^\infty$-microlocally regular at $(\widetilde x_0,\widetilde \xi_0)$.
\end{lemma}
\begin{proof}
$u$ is $\Mod^\infty$-microlocally regular at $(x_0,\xi_0)$ iff there exists $\phi\in\Mod_c^\infty(\ster\Omega_c)$ which satisfies eq.\ \eqref{eq-def-M-infty-mr}. Then clearly also
\begin{equation}\label{eq-phi-tilde}
\widetilde \phi(x)=1\text{ in }\GenC,\  \forall x\in\GenR^d, x\approx_\fast \widetilde x_0\text{ \ and \ }\Fourier(\widetilde \phi \widetilde u)(\xi)= 0\text{ in }\GenC,\ \forall \xi \in \GenR^d_{fs},\ \frac\xi{\abs\xi}\approx_\fast \widetilde \xi_0,
\end{equation}
i.e., $\widetilde u$ is $\widetilde\Gen^\infty$-microlocally regular at $(\widetilde x_0,\widetilde \xi_0)$.

Conversely, if $\widetilde\phi\in\Gen^\infty_c(\Omega)$ satisfies eq.\ \eqref{eq-phi-tilde}, then we can find a representative $\phi\in\Mod^\infty_c(\ster\Omega_c)$ which satisfies $\abs{\phi(x)}\ge 1/R$ and $\abs{\partial^\alpha \phi(x)}\le R$ for each $x\in\ster\R^d$, $x\approx_\fast x_0$ and $\alpha\in\N^d$, for some $R\approx 1$. Further, $\fourier{\phi u}(\xi)\approxeq 0$ for each $\xi\in\ster\R^d_{fs}\cap \ster\R^d_\Mod$ with $\frac\xi{\abs\xi}\approx_\fast \xi_0$. As $\fourier{\phi u}\in\Mod_\Schwartz(\R^d)$, also $\fourier{\phi u}(\xi)\approxeq 0$ for each $\xi\in\ext(\ster\R^d_\Mod)$. Let $n\in\N$ and $B_n:=\{\xi\in\ster\R^d: \abs{\xi}\ge \caninf^{-1/n}$, $\abs[\big]{\frac\xi{\abs\xi}- \xi_0}\le \caninf^{1/n}\}$. By Cor.\ \ref{cor-zero-on-union}, $\fourier{\phi u}(\xi)\approxeq 0$ for each $\xi \in B_n$. As $n\in\N$ is arbitrary, $\fourier{\phi u}(\xi)\approxeq 0$ for each $\xi\in\ster\R^d_{fs}$ with $\frac\xi{\abs\xi}\approx_\fast \xi_0$. Thus $u$ is $\Mod^\infty$-microlocally regular at $(x_0,\xi_0)$ by Prop.\ \ref{def-regular-with-cutoff}.
\end{proof}
Using this lemma, one easily recovers \cite[Thms.\ 4.5 and 5.3]{HVMicrolocal} about Colombeau generalized functions by factorization modulo negligible elements from Theorems \ref{proj-of-WF-is-singular-support} and \ref{consistency}.

\section*{Acknowledgment}
We are grateful to P.\ Giordano for pointing out references to earlier work on constructive versions of nonstandard principles.

\end{document}